\documentclass[12pt]{article}
\usepackage{latexsym}
\usepackage{amssymb}
\usepackage{amsmath}
\usepackage{amsthm}
\usepackage{mathabx}
\usepackage{rotating}
\usepackage{multirow}
\usepackage{mathrsfs}
\usepackage{wasysym}
\usepackage{enumerate}
\usepackage{nicefrac}
\usepackage{rawfonts}
\input{prepictex}
\input{pictex}
\input{postpictex}
\DeclareMathAlphabet{\zap}{OT1}{pzc}{m}{it}
\def\ZZ{\mathbb{Z}}

\def\CP{\mathbb{CP}}
\def\cpbar{\overline{\mathbb{CP}}_2}

\newtheorem{thm}{Theorem}[section]

\newtheorem{prop}[thm]{Proposition}

\def\ZZ{{\mathbb Z}}

\def\CP{{\mathbb C \mathbb P}}

\DeclareMathOperator{\Wh}{Wh}
\begin{document}

\title{Scalar Curvature
and  Product Manifolds}

\author{Claude LeBrun\thanks{Supported 
in part by NSF grant DMS-2203572  and  a Simons Fellowship.}\\ 
Stony
 Brook
 University} 
  
\date{}

\maketitle

\hspace{1.6in}
\begin{minipage}{3.4in}
\begin{quote} {\em In fond memory of  my late friend and wonderful   colleague
 Jim Simons.}
\end{quote}
\end{minipage}

\bigskip 

\begin{abstract} We display   $4$-dimensional counter-examples to a  restricted  form of Rosenberg's $S^1$-Stability 
Conjecture \cite{rosenberg}  recently proposed     by Jie  Xu \cite{jie}. Namely, we show that there are smooth compact $4$-manifolds $M$ with
Euler characteristic $\chi (M) =0$ that  do not admit positive-scalar-curvature  Riemannian metrics, but for which 
the corresponding Cartesian products $M\times S^1$  nevertheless do  admit   such metrics. 
\end{abstract}

\section{Introduction}
 Suppose that $M$ is a smooth closed $n$-manifold. If $M$ admits a positive-scalar-curvature Riemannian metric $g$,
 then  the Cartesian product  $M\times S^1$  admits positive-scalar-curvature  metrics, too, since   the   
 product metric $g+dt^2$ certainly   provides an obvious  example. But what about the ``downward''  direction? If $M\times S^1$
 admits metrics of scalar curvature $s>0$, do such metrics also exist on $M$? This question was codified by Jonathan 
 Rosenberg \cite[Conjecture 1.24]{rosenberg}, who called it the {\em $S^1$-Stability Conjecture}, even though  he  clearly indicated, in the very same paper, that  it fails when $n=4$.
 
This conjecture   may very well have been  inspired by the   Schoen-Yau  method of descent \cite{syrger},  which does at least allow one to  prove {\em something} in this downward direction. 
Let us  henceforth assume, for simplicity, that $M$ is oriented.
Then  for any metric $\widehat{g}$ on $M\times S^1$, geometric measure theory  \cite{federer}  guarantees that the homology class of $M\times \{ pt\} \subset M \times S^1$
is necessarily represented by a mass-minimizing current $\mathscr{Z}$.  If $n\leq 6$, this current can then be expressed \cite{lawson-montreal,simonsreg} as a sum
$$\mathscr{Z}=\sum n_j Z_j$$ 
of  disjoint  compact connected  smoothly  embedded 
oriented
$n$-dimensional hypersurfaces  $Z_j\subset M$ with positive integer multiplicities; and in fact, it was recently shown  \cite{chomanschu} 
that this even works  for $n\leq 10$ if  we also require  $\widehat{g}$
to belong to a certain open dense set of ``generic'' metrics. Because   $\mathscr{Z}$ is mass-minimizing, each  of these $n$-manifolds  is a stable minimal hypersurface  in $(M\times S^1 , \widehat{g})$, 
and the fact that  $\mathscr{Z}$ has homological intersection number $+1$ with $\{ pt\} \times S^1$ implies that at least one of them, 
say $Z=Z_j$, also has the property that the restriction $Z\to M$  of the first-factor projection $M\times S^1 \to M$ has 
 non-zero  degree. If $M\times S^1$ admits
positive-scalar-curvature metrics and $n\leq 10$,  then we can take our ``generic''  metric $\widehat{g}$ to have scalar curvature $s_{\widehat{g}}>0$. 
Letting  $h$ and $\gemini$  respectively denote  the induced metric and second-fundamental form of  $Z$,   Jim  Simons' second-variation formula \cite[Theorem 3.2.2]{simons} for the $n$-dimensional volume $\mathscr{A}$  of $Z$ then  tells us that 
$$
2 \mathscr{A}^{\prime \prime} (0) =\int_Z \left[2 |\nabla u|^2 + 
 (s_h-s_{\widehat{g}} -  |\gemini|^2)u^2\right] \,  d\mu_h
$$
for any smooth function $u\in C^{\infty} (Z)$, thought of as representing a normal variation vector field. The  volume-minimizing property  of ${Z}$ therefore   implies that 
\begin{equation}
\label{jiminy}
 \int_Z \left[2|\nabla u|^2 + 
s_h u^2\right]\, d\mu_h \geq  \int_Z s_{\widehat{g}} \,u^2\, d\mu_h > 0
\end{equation}
for every smooth positive function $u$ on $Z$. 
The inequality  \eqref{jiminy} then implies that $h$ is conformally equivalent to 
a metric of positive scalar curvature.\footnote{If $n=2$, apply \eqref{jiminy} to $u=1$ to show that  $\chi (Z) > 0$ by 
Gauss-Bonnet;   classical uniformization  then implies that the oriented  conformal manifold
$(Z, [h])$
must be  the standard $2$-sphere.  If $n\geq3$,    instead set $\widetilde{h}=u^{p-2}h$ for  $p =2n/(n-2)$, and notice that   \eqref{jiminy} implies that the total scalar curvature 
$\int s_{\widetilde{h}}d\mu_{\widetilde{h}} = \int [(p+2) |\nabla u|^2 + s_h u^2 ]d\mu_h$ 
of $(Z, \widetilde{h})$ 
is  positive for any metric $\widetilde{h}$ in the conformal class $[h]$ of $h$; 
the claim therefore follows, since
 any conformal class contains metrics for which the scalar curvature  does not change sign.} 
 When $n\leq 10$, this  shows that there is a closed oriented Riemannian  $n$-manifold $Z$ of positive scalar curvature
and a smooth map $Z\to M$ of non-zero degree.
 For example,  in   dimension  $n=3$, where the existence of positive-scalar-curvature metrics 
 only depends \cite{bbbpm,gvln2,lott,sy3man} on  homotopy type, this    allows one to conclude that the $S^1$-stability conjecture holds. 
However, in dimension $n=4$, where  the existence of positive-scalar-curvature metrics instead   delicately depends on diffeotype rather than on  homotopy type, 
 the weak sort of information that can be deduced from the method of descent really does  nothing   to  impede  the failure of the conjecture. 
 
 \smallskip 
 
The  main  purpose of the present paper is to provide a  clear negative answer to a question asked by Jie Xu \cite{jie}, who has proposed  a
restricted form of the $S^1$-stability conjecture  that limits  the question to manifolds of  Euler characteristic zero. However, in dimension
$n=4$, we will  see that  counter-examples to even this restricted conjecture still  exist in  profusion:

\begin{thm} \label{melon} 
There are smooth, closed $4$-manifolds $M$ with  $\chi (M)=0$ that do not admit  metrics of positive scalar curvature, but for which the  corresponding products $M\times S^1$
 nevertheless {do}  admit  such  metrics.
\end{thm}

\noindent 
Here the condition $\chi (M)=0$ of course excludes simply-connected $4$-manifolds, so the proof of Theorem  \ref{melon}  necessarily   involves various   technical tricks that will
 make  it  seem a bit  less main-stream
than the arguments cited  in  \cite{rosenberg}. 

\smallskip

But my real motive for choosing to tell this particular story in the present  context  is that  I believe that it's a tale   that 
would have interested my friend and larger-than-life colleague Jim Simons.  I hope that  the above discussion makes it  clear that 
Jim's  mathematical contributions play a natural and  important role in almost any discussion of this subject. 
But it also gives me a chance to remind people of Jim's amazing role, as departmental  chair,  in creating the 
Stony Brook Mathematics Department, and turning it  into a powerhouse in geometry and topology. The arguments I will describe here crucially depend on  the results of 
 several  of the people that Jim brought to Stony Brook, including  Misha Gromov, Blaine Lawson, and Lowell Jones.
I will also make key use of a  result discovered by  Jimmy Petean, one of my star Stony Brook PhD students, thereby  
illustrating  how my long apprenticeship at Stony Brook has allowed me to learn wonderful geometric ideas from our
amazing graduate students as well as from my senior colleagues. I  thus feel incredibly lucky to have
spent much of my life at Stony Brook, where the air has always been amazingly full of geometry and topology. And I am 
therefore deeply grateful to Jim, whose efforts did so much to create and maintain Stony Brook's vibrant  and stimulating mathematical atmosphere.

\smallskip

Rest in peace, Jim. You are profoundly  missed here,  by all of us! 

\pagebreak 
 
 \section{A Four-Dimensional Counter-Example} \label{core}
 
 Let's begin the proof of Theorem \ref{melon} by carefully discussing a single example. 
 
 \medskip
 
Let $X$ be a smooth  quintic complex  hypersurface in $\CP_3$, which  we could, for example,  take
to be the Fermat quintic 
$$X:= \left\{ [t:u:v:w]\in \CP_3~|~ t^5+u^5+v^5+w^5 =0\right\}.$$
By the Lefschetz hyper-plane section theorem, this smooth compact complex surface is simply connected, 
and by applying adjunction to calculate its Chern and Pontryagin classes, one can easily show  that it has 
 Euler characteristic $\chi (X)= 55$ and signature $\tau (X)= -35$. Freedman's topological classification of simply connected $4$-manifolds \cite{frequin} therefore
 implies that $X$  is homeomorphic to $9 \CP_2 \# 44 \overline{\CP}_2$.  The fact that these two smooth oriented $4$-manifolds are however not  {\em diffeomorphic}   
  was originally proved  using    Donaldson invariants and  the anti-self-dual Yang-Mills equations \cite{don}, although the simpler Seiberg-Witten proof  \cite{witten}  of this fact is more 
  relevant to 
  the following discussion.

 \medskip 
 
  Now let
$Y = X\# \cpbar$ be the blow-up of $X$ at one point, and notice that this non-minimal complex surface of general type \cite{bpv} 
is therefore  homeomorphic to $9 \CP_2 \# 45 \cpbar$, and so in particular has 
 Euler characteristic $\chi = 56$. Finally,  set 
\begin{eqnarray*}
M &=& Y \# 28 (S^1 \times  S^3),\\
N &=& 9 \CP_2 \# 45 \cpbar  \# 28 (S^1 \times  S^3). 
\end{eqnarray*}
It is then easy to see that $M$ and $N$  are again homeomorphic, and that 
$$\chi (M) = \chi (N) = 0.$$
We will now show that $M$ provides a $4$-dimensional counter-example to Xu's restricted $S^1$-stability conjecture, by proving that:

\begin{enumerate}[(i)]
\item $M$ does not admit metrics of positive scalar curvature; while \label{primo} 

\item   $N$ admits metrics of positive scalar curvature; but \label{secondo} 

\item $M \times S^1$ is diffeomorphic to $N \times S^1$,  and therefore  admits metrics of positive scalar curvature. \label{terzo} 
\end{enumerate}

In fact, our proof of \eqref{primo} will actually prove much more:

\begin{prop} 
\label{uno}
The smooth compact $4$-manifold $M$ has negative Yamabe invariant $\mathscr{Y}(M)= - 4 \pi \sqrt{10}$.
Moreover, the  minimax defining this invariant is unachieved, so  that the scalar curvature of any unit-volume metric on M
must  in particular  have  scalar curvature   $<  - 4 \pi \sqrt{10}$ at some point. 
\end{prop}

\begin{proof}
The {\em Yamabe invariant} \cite{lno}  
of a closed  $n$-manifold $M$ is    defined  to be 
\begin{equation}
\label{yamdef}
\mathscr{Y}(M) = \sup_\gamma \inf_{g\in \gamma} \frac{\int_M s_g~d\mu_g}{[ \int_Md\mu_g]^{(n-2)/n}} 
\end{equation}
where $\gamma$ varies over all conformal classes of Riemannian metrics $g$ on $M$; this invariant was originally  introduced  independently  by Kobayashi \cite{okob}
 and Schoen \cite{sch},  who gave it  entirely different  names. It is positive
iff the manifold admits positive-scalar-curvature metrics. 
For the underlying $4$-manifold of any  compact complex surface of general  type, such as the quintic $X$ or its one-point blow-up $Y$,  the Yamabe invariant was systematically calculated 
in \cite[Theorem 7]{lno}, using  a combination of Seiberg-Witten theory and concrete geometric constructions;  it was  shown to be negative, 
to be unchanged by  blowing up,  and to be given by $-4\pi\sqrt{2c_1^2}$ of the minimal model. In our case, we therefore have 
$$\mathscr{Y}(Y)  =\mathscr{Y}(X\# \cpbar)=\mathscr{Y}(X) =-4\pi\sqrt{2 c_1^2 (X)}=-4\pi\sqrt{10}.$$
Indeed, because our  $X$ has $c_1< 0$, it carries a negative-Einstein-constant K\"ahler-Einstein metric \cite{aubin,yau}, and by 
\cite{lno}, this metric actually realizes the Yamabe minimax 
on $X$. On the other hand, \cite[Theorem 7]{lno} guarantees that 
the Yamabe minimax is unattained 
   on the non-minimal complex surface $Y$. 

However,  Petean \cite{jp2} then gave  an  argument  based on surgery that  proved that taking the connected sum  of a  Yamabe-negative  $4$-manifold  with $S^1\times S^3$ does not change  the 
Yamabe invariant. By induction, we therefore have
$$
\mathscr{Y}(M) = \mathscr{Y}(Y \# 28 [S^1 \times  S^3])= \mathscr{Y}(Y) = \mathscr{Y}(X)=-4\pi\sqrt{10}.
$$
Moreover, this Yamabe invariant cannot be attained by any metric; if it were, the  Yamabe metric $g$  attaining 
the minimax  would be Einstein, and this would then contradict the Hitchin-Thorpe inequality \cite{hit,tho},   since 
$$(2\chi + 3 \tau )(M) =  - 108 < 0.$$
The Yamabe constant of every conformal class on $M$ must therefore be less than $\mathscr{Y}(M)$, and the minimum of the scalar curvature of any 
unit-volume metric $g$ on $M$ must therefore be less than $\mathscr{Y}(M)=-4\pi\sqrt{10}$.
\end{proof}


\medskip

\noindent 
Next, we   put our argument on a firm footing by  proving  assertion \eqref{secondo}:

\begin{prop} 
\label{due}
The $4$-manifold $N = 9 \CP_2 \# 45 \cpbar  \# 28 (S^1 \times  S^3)$ admits
Riemannian metrics of positive scalar curvature. 
\end{prop} 
\begin{proof} From the expression of   $N$   as a connected sum, we immediately see   that it is obtained from a disjoint union 
of copies of $\CP_2$, $\cpbar$, and $S^1 \times S^3$ by performing surgery in codimension $4$. 
Since the  these building-blocks all admit standard metrics of positive scalar curvature, the basic codimension $\geq 3$ surgery 
 result   of Gromov and Lawson \cite[Theorem A]{gvln}  therefore produces  a positive-scalar-curvature 
metric on their connected sum. 
\end{proof} 


\noindent 
Finally, we  conclude our argument by  proving assertion \eqref{terzo}:

\begin{prop}  
\label{tre} 
The $5$-manifolds $M\times S^1$ and $N\times S^1$ are diffeomorphic. 
\end{prop}
\begin{proof}
A fundamental theorem due to Wall \cite{wall} asserts that if  two smooth simply-connected closed   $4$-manifolds are homotopy equivalent, 
then they are 
actually $h$-cobordant. For example,  this means that there is a smooth compact oriented $5$-manifold $V$ whose boundary, as an oriented but disconnected $4$-manifold,  is diffeomorphic to the 
disjoint union of $\overline{Y}$  and $9 \CP_2 \# 45 \cpbar$, and such  that the inclusion of either  boundary component into $V$ is a homotopy equivalence. 
Now choose $28$ disjoint embedded arcs in $V$  from $\overline{Y}$  to $9 \CP_2 \# 45 \cpbar$  that meet both 
boundary components transversely. Next, thicken these up into $28$ disjoint embedded copies of $D^4 \times [0,1]$, thereby  interpolating between $28$ embedded $4$-disks in each 
boundary component. Finally, remove the interior of each copy of $D^4 \times [0,1]$, and glue in a copy of $[(S^3 \times S^1) - D^4]\times [0,1]$ in its place. 
The resulting oriented $5$-manifold $W$ has boundary diffeomorphic to a disjoint union of $\overline{M}=\overline{Y \# 28 (S^3 \times S^1)}$ and $N= 9 \CP_2 \# 45 \cpbar  \# 28 (S^1 \times  S^3)$,
and  the inclusion of either boundary component into $W$ is a homotopy equivalence. 
Thus,  $W$ is an $h$-cobordism  between the oriented $4$-manifolds $M$ and $N$. 
In particular, 
$$
\pi_1(M)\cong \pi_1 (W)\cong \pi_1(N) \cong \underbrace{\ZZ \ast \cdots \ast \ZZ}_{28}=:\mathsf{F},
$$
where, for later convenience, we have introduced the  abbreviated notation $\mathsf{F}$ for the free group on $28$  generators.

This also  immediately implies  that  the $6$-manifold $W\times S^1$ is    an  $h$-cobordism between $M\times S^1$ and 
$N\times S^1$. We will now show that  $W\times S^1$ is in fact an $s$-cobordism, meaning that the inclusion
of either boundary component into $W$ is a {\em simple} homotopy equivalence. In principle, the obstruction to this \cite{kerv-scobo} is   
an element of the Whitehead group $\Wh (\pi_1 (W))=\Wh (  \mathsf{F} \times\ZZ )$, where 
we have once again used 
$\mathsf{F}$ to denote  the free group on $28$ generators. But  notice 
that $\mathsf{F}$ is also the fundamental group of $\CP_1 -\{ 29~pts\}$, and  that this $29$-times  punctured Riemann-sphere actually 
carries a family\footnote{A particularly simple example of such a metric can be 
obtained by  pulling back 
the hyperbolic metric on the thrice-punctured $2$-sphere via an unbranched cyclic cover of order $27$; the fact that this example
has area $54\pi$ 
 then becomes apparent to the naked eye, because the resulting hyperbolic surface 
has  an obvious  tiling by  $54$ ideal triangles.} of complete, finite-area metrics of Gauss curvature $K=-1$. 
However, a striking  result of Farrell and Jones \cite[Corollary 2]{farjo} says that if $\mathsf{G}$ is the fundamental group of 
a complete, finite-volume  Riemannian manifold of  pinched\footnote{That is, with sectional curvature in $[-b,-a]$ for some $b>a >0$.}
negative sectional curvature, then $\Wh (\mathsf{G}) = \Wh (\mathsf{G}\times \ZZ) =0$.
As an immediate consequence, we therefore have
$$
\Wh (\pi_1 (W))=\Wh (  \mathsf{F} \times\ZZ )=0,
$$
so that $W$ is therefore an $s$-cobordism. The $s$-cobordism theorem \cite{kerv-scobo,s-cobo} therefore guarantees that 
$W$ is diffeomorphic to $(M\times S^1) \times  [ 0,1]$,  and it therefore follows   that $M\times S^1$ and $N\times S^1$ are actually  diffeomorphic. 
\end{proof}

Since $N$ admits  metrics $g$ with $s >0$ by Proposition \ref{due}, it follows that $N\times S^1$ also  admits metrics, such as $\hat{g}= g+dt^2$, 
that have positive scalar curvature.
 But since  $M\times S^1$  is diffeomorphic to  $N\times S^1$ by Proposition \ref{tre}, it follows that $M\times S^1$ admits 
$s>0$ metrics, too. But since $M$ has $\chi (M) =0$ and does not carry metrics of positive scalar curvature by Proposition \ref{uno},
this example therefore fulfills all the desiderata promised by Theorem \ref{melon}. 

\section{Concluding Remarks}

 So far, we have just produced   one example  fulfilling all the condition promised by Theorem \ref{melon},
so one  might reasonably object that the Theorem actually  seems to   promise many of them. 
However, such examples do in fact  exist in considerable profusion.   For example,  we could let $X$ be any simply-connected 
minimal  complex surface of 
general type,  blow it up at a point  to guarantee that the resulting manifold is non-spin, and then
blow it up again, if needed, in order to  guarantee that the resulting non-minimal complex surface $Y$ has even
Euler characteristic $\chi (Y) =2k$. The resulting  $Y$ is then  $h$-cobordant  to 
(and, incidentally,  homeomorphic to)  $\ell  \CP_2 \# m \cpbar$ for appropriate  odd integers $\ell$ and $m$ with 
$\ell + m=2(k-1)$. If we now  set $M=Y \# k (S^1 \times S^3)$ and $N=\ell  \CP_2 \# m \cpbar \# k (S^1 \times S^3)$, 
the same arguments we used in \S \ref{core} then show that $M$ and $N$ are $h$-cobordant. We also similarly  have 
$$\mathscr{Y}(M) = \mathscr{Y}(Y) =  \mathscr{Y} (X)= -4\pi \sqrt{c_1^2 (X)} < 0,$$
where   $c_1^2 (X) > 0$ by of our assumption that $X$ is minimal and of general type  \cite{bpv}. (Again, 
the Yamabe minimax is unachieved on $M$, since we have   arranged that $\chi (M)=0$ and $\pi_2(M)\neq 0$,
which together  guarantee \cite{hit} that $M$ cannot admit an Einstein metric.)
On the other hand, the  Gromov-Lawson surgery argument  used in \S \ref{core}  similarly  shows that $N$ again admits metrics of positive scalar curvature. 
Finally,  replacing the  previous free group  $\mathsf{F}$ with 
$$\mathsf{F}:=  \underbrace{\ZZ \ast \cdots \ast \ZZ}_{k},$$
 which is  the fundamental group of $M$, $N$ and an $h$-cobordism $W$ between them, as well as  of the punctured $2$-sphere $\CP_1 -\{ (k+1)~pts\}$, which 
admits a complete $K=-1$ metric of area $2\pi (k-1) < \infty$. The previously-cited result of Farrell and Jones \cite{farjo} therefore
 once again shows that $$\Wh (\pi_1 (M\times S^1))=\Wh (\mathsf{F} \times \ZZ)=0,$$  so  that  $M\times S^1$ and $N\times S^1$
are   once again 
$s$-cobordant, and therefore  diffeomorphic. Taking $X$ to be any  smooth $c_1 < 0$ complete-intersection surface  in a weighted projective space
then certainly shows that infinitely many values of the Yamabe invariant $\mathscr{Y}(M)$ and the signature $\tau (M)$ are possible in this construction. 
More exotic constructions \cite{ruler} 
of minimal simply-connected complex surfaces of general type $X$ 
could instead be used, and this alone 
 begins to    indicate how we have only  scratched the surface of what is possible.

%
%

Finally, while the  surgical  result of Petean \cite{jp2} that we have used here is a  natural extension of 
the ideas  of Gromov and Lawson \cite{gvln}, and moreover leads to striking results 
\cite{jp3} in higher dimensions, 
the obstruction to positive scalar curvature  at play in these examples can also be  detected directly via an interesting  variation 
\cite{rio,lric,ozsz} on the standard Seiberg-Witten invariant. 
Readers interested in generalizing the present construction through the use of, say,    symplectic $4$-manifolds or properly elliptic complex surfaces 
should  therefore  by all means feel invited  to explore this compelling circle of  ideas.

\pagebreak 
%

\end{document}